\documentclass{amsart}
\usepackage{graphics}
\usepackage{ae}
\usepackage{amsmath,bbm}

\usepackage{amssymb}

\numberwithin{equation}{section}

\newtheorem{theo}{Theorem}
\newtheorem{prop}[theo]{Proposition}

\newtheorem{coro}[theo]{Corollary}
\newtheorem{lemma}[theo]{Lemma}
\theoremstyle{definition}
\newtheorem{defi}[theo]{Definition}

\begin{document}

\title[Decidability of uniform recurrence of morphic sequences]{Decidability of uniform recurrence of morphic sequences}
\author{Fabien Durand}
\address[F.D.]{\newline
Universit\'e de Picardie Jules Verne\newline
Laboratoire Ami\'enois de Math\'ematiques Fondamentales et
Appliqu\'ees\newline
CNRS-UMR 7352\newline
33 rue Saint Leu\newline
80039 Amiens Cedex 01\newline
France.}
\email{fabien.durand@u-picardie.fr}

\begin{abstract}
We prove that the uniform recurrence of morphic sequences is decidable. 
For this we show that the number of derived sequences of uniformly recurrent morphic sequences is bounded.
As a corollary we obtain that uniformly recurrent morphic sequences are primitive substitutive sequences. 
\end{abstract}

\maketitle

\section{Introduction}

In this paper we answer a question raised in \cite{Nicolas&Pritykin:2009}.
We prove the decidability of the following problem:

\medskip

{\bf Input:}
Two finite alphabets $A$ and $B$, an endomorphism $\sigma : A^* \to A^*$ prolongable on the letter $a$, a morphism $\phi : A^* \to B^*$.

{\bf Question:}
Is $\phi (\sigma^\infty (a))$ uniformly recurrent?

\medskip

Thus, the main result of the paper is the following.

\begin{theo}
\label{theo:main}
The uniform recurrence of morphic sequences is decidable.
\end{theo}

During the writing process of this work, I. Mitrofanov put on arxiv another solution of this problem \cite{Mitrofanov:preprint2011b}. 
Observe that Mitrofanov in \cite{Mitrofanov:preprint2011a} and Durand in \cite{Durand:preprint2011} also recently solved the HD0L ultimate periodicity problem. 
The proofs of Mitrofanov and Durand find their main arguments in \cite{Kanel-Belov&Mitrofanov:preprint2011} and \cite{Durand:1998}, respectively.
It seems (the first article is written in Russian) that these two articles share some common features that would be very interesting to understand.   

To prove the decidability of uniform recurrence for morphic sequences we proceed as in \cite{Durand:2012} and \cite{Durand:preprint2011}, that is, we use the main result of \cite{Durand:1998} that we extend to our setting.

\begin{theo}
\label{theo:maindurand1998}
Let $x$ be a uniformly recurrent sequence defined on a finite alphabet.
Then, it is a morphic sequence if and only if the set of its derived sequences (on prefixes of $x$) is finite.
\end{theo}

In fact we prove more giving effective bounds.
As a consequence this shows that uniformly recurrent morphic sequences are primitive substitutive sequences.

\begin{theo}
\label{theo:morphic-sub}
Uniformly recurrent morphic sequences are primitive substitutive sequences. 
\end{theo}

As we use derived sequences, we also use return words.
Because we have to deal with some non necessarily uniformly recurrent sequences we need to introduce a notion of return words to sets of words.

\medskip

This paper is composed as follows.
In Section \ref{section:background} we recall some definitions and well-known facts. 
Moreover, we introduce return words to sets of words and give some first properties. 
This construction is fundamental for our decidability algorithm.
In Section \ref{section:bounds} we prove Theorem \ref{theo:maindurand1998} giving explicit bounds concerning this notion of return words in the context of morphic sequences. 
It is these bounds that make our problem decidable.
In Section \ref{section:decid} we show the main result of this paper: Theorem \ref{theo:main}.
We end with Section \ref{section:remarks} where Theorem \ref{theo:morphic-sub} is proven and Theorem \ref{theo:maindurand1998} is extended to descendant sequences (see \cite{Holton&Zamboni:1999}).
These results are interpreted in terms of dynamical systems, few words are given on the algorithmical complexity, and we add open decidability problems for morphic sequences to the list given in \cite{Nicolas&Pritykin:2009}.

\section{Some background}
\label{section:background}

\subsection{Words and sequences}
 
An {\it alphabet} $A$ is a finite set of elements called {\it
  letters}. 
A {\it word} over $A$ is an element of the free monoid
generated by $A$, denoted by $A^*$. 
Let $x = x_0x_1 \cdots x_{n-1}$
(with $x_i\in A$, $0\leq i\leq n-1$) be a word, its {\it length} is
$n$ and is denoted by $|x|$. 
The {\it empty word} is denoted by $\epsilon$, $|\epsilon| = 0$. 
The set of non-empty words over $A$ is denoted by $A^+$. 
The elements of $A^{\mathbb{N}}$ are called {\it sequences}. 
If $x=x_0x_1\cdots$ is a sequence (with $x_i\in A$, $i\in \mathbb{N}$) and $I=[k,l]$ an interval of
$\mathbb{N}$ we set $x_I = x_k x_{k+1}\cdots x_{l}$ and we say that $x_{I}$
is a {\it factor} of $x$.  If $k = 0$, we say that $x_{I}$ is a {\it
 prefix} of $x$. 
The set of factors of length $n$ of $x$ is written
$\mathcal{L}_n(x)$ and the set of factors of $x$, or the {\it language} of $x$,
is denoted by $\mathcal{L}(x)$. 
We set $p_x (n)= \# \mathcal{L}_n(x)$.
It is the word complexity (function) of $x$. 
The {\it occurrences} in $x$ of a word $u$ are the
integers $i$ such that $x_{[i,i + |u| - 1]}= u$. 
If $u$ has an occurrence in $x$, we also say that $u$ {\em appears} in $x$.
When $x$ is a word,
we use the same terminology with similar definitions.

The sequence $x$ is {\it ultimately periodic} if there exist a word
$u$ and a non-empty word $v$ such that $x=uv^{\infty}$, where
$v^{\infty}= vvv\cdots $.
In this case $|v|$ is called a {\em length period} of $x$. 
It is {\it periodic} if $u$ is the empty word. 
A word $u$ is {\em recurrent} in $x$ if it appears in $x$ infinitely many times.
A sequence $x$ is {\it uniformly recurrent} if every factor $u$ of $x$
appears infinitely often in $x$ and the greatest
difference of two successive occurrences of $u$ is bounded.

In the sequel $A$ and $B$ will always be finite alphabets.

\subsection{Morphisms and matrices} 
Let $\sigma$ be a {\it morphism} 
from $A^*$ to $B^*$. 
When $\sigma (A) = B$, we say $\sigma$ is a {\em coding}. 
Thus, codings are onto.
If $\sigma (A)$ is included in $B^+$, it induces by concatenation a map
from $A^{\mathbb{N}}$ to $B^{\mathbb{N}}$. These two maps are also called $\sigma$.
To the morphism $\sigma$ is naturally associated the matrix $M_{\sigma} =
(m_{i,j})_{i\in B , j \in A }$ where $m_{i,j}$ is the number of
occurrences of $i$ in the word $\sigma(j)$.
We set $|\sigma | = \max_{a\in A} |\sigma (a)|$ and $\langle \sigma \rangle = \min_{a\in A} |\sigma (a)|$.

\subsection{Morphic and substitutive sequences}

In the sequel we use the definition of substitution that is in \cite{Queffelec:2010} and the notion of substitutive sequence defined in \cite{Durand:1998}.
Both are restrictive. 
For non restrictive context we will speak of (prolongable) endomorphisms and (purely) morphic sequences.

The language of the endomorphism $\sigma : A^* \to A^*$ is the set $\mathcal{L} (\sigma ) $ of words appearing in some $\sigma^n (a)$, $a\in A$.
The number of words of length $n$ in $\mathcal{L} (\sigma )$ is denoted by $p_\sigma (n)$.

If there exist a letter $a\in A$ and a word $u\in A^+$ such that $\sigma(a)=au$ and moreover, if $\lim_{n\to+\infty}|\sigma^n(a)|=+\infty$, then $\sigma$ is said
    to be {\em prolongable on $a$}.
It is a {\em substitution} whenever it is prolongable and growing (that is, $\lim_n \langle \sigma^n \rangle = +\infty$). 
Since for all $n\in\mathbb{N}$, $\sigma^n(a)$ is a prefix of $\sigma^{n+1}(a)$ and because
     $|\sigma^n(a)|$ tends to infinity with $n$, the
     sequence $(\sigma^n(aaa\cdots ))_{n\ge 0}$ converges (for the usual
     product topology on $A^\mathbb{N}$) to a sequence denoted by
     $\sigma^\infty(a)$. 
The endomorphism $\sigma$ being continuous for the product topology, $\sigma^\infty(a)$ is a fixed point of $\sigma$: $\sigma (\sigma^\infty(a)) = \sigma^\infty(a)$.
A sequence obtained in
    this way by iterating a prolongable endomorphism is said to be {\em purely morphic} (w.r.t. $\sigma$) or {\em purely substitutive} when $\sigma$ is a substitution. 
If $x\in A^\mathbb{N}$ is purely morphic and $\phi:A^*\to B^*$ is a morphism then the sequence $y=\phi (x)$ is said to be a {\em morphic sequence} (w.r.t. $\sigma$). 
When $\phi $ is a coding and $\sigma $ a substitution, we say $y$ is {\em substitutive} (w.r.t. $\sigma$). 

Whenever the matrix associated to $\sigma $ is
primitive (i.e., when it has a power with strictly positive coefficients) we say that $\sigma$ is a {\it primitive endomorphism}.  
In this situation we easily check that $\mathcal{L} (\sigma ) = \mathcal{L} (\sigma^\infty (a))$.
A sequence is {\em primitive substitutive} if it is substitutive w.r.t. a primitive substitution.

We
say $\sigma$ is
{\it erasing} if there exists $b\in A$ such that $\sigma (b)$ is the
empty word.  
Such a letter is called an {\em erasing letter}.

\subsection{Iterated matrices}

For the sequel we need the following well-known results on iterated matrices.

\begin{lemma}
\label{lemme:horn}
\cite{Horn&Johnson:1990}
If $M$ is a primitive $d\times d$ matrix then there exists $k\leq d^2-2d+2$ such that $M^k $ has strictly positive entries.
\end{lemma}

>From Lemma \ref{lemme:horn}, Section 4.4 and Section 4.5 in \cite{Lind&Marcus:1995} we deduce following proposition. 

\begin{prop}
\label{prop:decomprim}
Let $\sigma : A^* \to A^*$ be an endomorphism.
Then, there exist three positive integers 
$r_\sigma $, $q$ and $l$, where $q\leq l-1$, and a partition 
$\{ A_i | 1\leq i\leq l \}$ of $A$ such that
\begin{equation*}
M_\sigma^{r_\sigma}=\bordermatrix{        & A_1     & A_2         & \cdots & A_q       & A_{q+1} & A_{q+2} & \cdots & A_{l}  \cr
                  A_1     & M_1     & 0           & \cdots & 0         & 0       & 0       & \cdots & 0      \cr
                  A_2     & M_{1,2} & M_2         & \cdots & 0         & 0       & 0       & \cdots & 0      \cr
                  \vdots  & \vdots  & \vdots      & \ddots & \vdots    & \vdots  & \vdots  & \vdots & \vdots \cr
                  A_q     & M_{1,q} & M_{2,q}     & \cdots & M_q       & 0       & 0       & \cdots & 0      \cr
                  A_{q+1} & M_{1,q+1} & M_{2,q+1} & \cdots & M_{q,q+1} & M_{q+1} & 0       & \cdots & 0      \cr
                  A_{q+2} & M_{1,q+2} & M_{2,q+2} & \cdots & M_{q,q+2} & 0       & M_{q+2} & \cdots & 0      \cr
                  \vdots  & \vdots    & \vdots    & \ddots & \vdots    & \vdots  & \vdots  & \ddots & \vdots \cr
                  A_l     & M_{1,l}   & M_{2,l}   & \cdots & M_{q,l}   & 0       & 0       & \cdots & M_l    \cr},
\end{equation*}
where the matrices $M_i$ have only positive entries or are equal to zero.
Moreover, the partition and $r_\sigma$ can be algorithmically computed.
\end{prop}

Using the terminology of the previous proposition we will say that $M_\sigma$ has $l$ primitive components. 
When the substitution $\sigma$ is primitive and defined on an alphabet with $d$ letters, we get $q=l=1$ and $r_\sigma = d^2-2d+2$ (Lemma \ref{lemme:horn}).

In order to introduce the next definition, we need the following well-known result.

\begin{prop}
\label{prop:croissance}
Let $\sigma : A^* \to A^*$ be a non-erasing endomorphism.
There exists a computable integer $p$ such that for $\tau = \sigma^p$ and for all $b\in A$, there exist $\theta (b) \geq 1$, $d(b)\in{\mathbb N}$ and $c(b)\in \mathbb{R}$ such that

$$
\lim_{n\to+\infty} \frac{|\tau^n(b)|}{c(b)\, n^{d(b)}\, \theta(b)^n}=1.
$$

Moreover, for all $b\in A$ and all integers $d'\in [0,d(b)]$ there exists a letter $b'$, appearing in some $\sigma^k (b)$, such that 
$\lim_{n\to+\infty} \frac{|\tau^n(b')|}{c(b')\, n^{d'}\, \theta(b')^n}=1$, i.e., $d(b')  =d'$.
\end{prop}

\begin{proof}
A proof can be found in \cite{Durand&Rigo:2009} except that it is not indicated that $p$ is computable. 
But from this proof it is clear that $p$ can be chosen to be $r_\sigma$ (Proposition \ref{prop:decomprim}).
\end{proof}

This justifies the following definitions.

\begin{defi} 
Let $\sigma : A^* \to A^*$ be a non-erasing endomorphism. 
For all $a\in A$, if there exists a couple  $(d(a), \theta(a))$ satisfying 

$$
\lim_{n\to+\infty} \frac{|\sigma^n(a)|}{c(a)\, n^{d(a)}\, \theta(a)^n}=1,
$$

for some constant $c(a)$, then we say $(d(a), \theta(a))$ is the {\em growth type} of $a$ (w.r.t. $\sigma$),
and, $\theta (a)$ is called the {\em growth rate} of $a$ (w.r.t. $\sigma$). 
When the growth type of all the letters of a word $w$ exists, we say that the {\em growth type} of $w$ is the maximal growth type of its letters.
If
$(d,\theta)$ and $(e,\beta)$ are two growth types we say that
$(d,\theta)$ is {\em less than} $(e,\beta)$ (or $(d,\theta) <
(e,\beta)$) whenever $\theta < \beta$ or, $\theta = \beta$ and $d<e$.
\end{defi}

We say that $a\in A$ is a {\em growing letter}\index{growing letter}
if $\lim_{n\to +\infty} |\sigma^n (a)| = +\infty $.

Now suppose that all letters define a growth type. 
Notice that, taking a (computable) power a $\sigma$ if needed, this is always possible (Proposition \ref{prop:croissance}). 

We set $\alpha := \max \{ \theta(a) \mid a\in A \}$, $D := \max \{ d(a) \mid 
a\in A, \theta(a) = \alpha  \}$ and $A_{max} := \{a\in A \mid
\theta (a) = \alpha , d(a) = D \}$.  
We will say that the letters of $A_{max}$ are {\em of
  maximal growth}, $(D,\alpha)$ is the {\em growth type} of
$\sigma$ and $\alpha$ is the {\em growth rate} of $\sigma$. 
Observe that if $\alpha=1$, then in view of the last part of Proposition
\ref{prop:croissance}, there exists at least one non-growing letter (of
growth type $(0,1)$).  
Otherwise stated, if a letter has a polynomially bounded
growth, then there exists at least one non-growing letter.
Consequently $\sigma$ is growing ({\em i.e.}, all its letters are growing)
if and only if $\theta (a) > 1$ for all $a\in A$.  
We observe that for all $k\geq 1$, the growth type of $\sigma^k $ is $(D, \alpha^k )$.

\subsection{Return words}
\label{subsec:returnwords}

Let $y\in A^\mathbb{N}$ and $U$ be a finite set of words in $\mathcal{L} (y)$.
We call {\it return word to $U$} any word
$y_{[i,j-1]} = y_i y_{i+1} \cdots y_{j-1}$, where 

\begin{enumerate}
\item
$i$ is an occurrence of some word of $U$ in $y$;
\item
$j$ is an occurrence of some word of $U$ in $y$;
\item
there is no other occurrence between $i$ and $j$ of any word in $U$.
\end{enumerate}

The set of return words to $U$ is denoted by ${\mathcal{R}}_{y,U}$. 
When $U$ consists of a single word then we recover the classical notion of return words (see \cite{Durand:1996,Ferenczi&Mauduit&Nogueira:1996,Durand:1998}).
If $y$ is uniformly recurrent then the set of return words to $u$ is finite.

For classical return words, we enumerate the elements of $\mathcal{R}_{y,U}$ in the order of their first appearance in $y$.
Here we should precise how we will proceed because it could happen that a return word is followed, for different occurrences, by different words of $U$. 
Thus in some sense this return word represents two different return words (because it returns to two different elements of $U$).
We set

$$
\tilde{\mathcal{R}}_{y,U} = \{ (w,u) | wu\in \mathcal{L} (y) , w \in {\mathcal{R}}_{y,U}, u\in U \} . 
$$ 
 
We enumerate the elements $(w,u)$ in $\tilde{\mathcal{R}}_{y,U}$ (resp. $w \in \mathcal{R}_{y,U}$) in the order of the first appearance of $wu$ in $y$ (resp. of $w$ in $y$).
This defines a bijective map

$$
\tilde{\Theta}_{y,U} : \tilde{R}_{y,U}  \rightarrow \tilde{\mathcal{R}}_{y,U} \hbox{ (resp. } \Theta_{y,U} : R_{y,U} \rightarrow \mathcal{R}_{y,U} )
$$

where $\tilde{R}_{y,U} = \{ 1, 2, \dots \}$ and $\# \tilde{R}_{y,U} = \#\tilde{\mathcal{R}}_{y,U}$ (resp. $R_{y,U} = \{ 1, 2 \dots \}$ and  $\# R_{y,U} = \#\mathcal{R}_{y,U}$).
The word $\tilde{\Theta}_{y,U} (i) =(w,u)$ is such that $wu$ is the $i$-th such word occurring in $y$ and $\Theta_{y,U} (i)$ is the $i$-th return word to $U$ occurring in $y$.
 
Let $\delta_U : \tilde{\mathcal{R}}_{y,U}^* \to A^*$ be defined by $\delta_U (w,u) = w$. 

\begin{prop}
\label{prop:def-DU}
Suppose $y\in A^{\mathbb{N}}$ has infinitely many occurrences of elements of $U$ and has a prefix belonging to $U$.
Then, 

\begin{enumerate}
\item
There exists a unique sequence $\Delta_U (y) =(w_n,u_n)_n\in \tilde{\mathcal{R}}_{y,U}^\mathbb{N}$ such that the word
$w_0 w_1\cdots w_n u_n$ is a prefix of $y$;
\item
There exists a unique sequence $\tilde{\mathcal{D}}_U (y)\in \tilde{R}_{y,U}^\mathbb{N}$ satisfying 
$\tilde{\Theta}_{y,U} (\tilde{\mathcal{D}}_U (y)) = \Delta_U (y)$;
\item
There exists a unique sequence $\mathcal{D}_U (y)\in R_{y,U}^\mathbb{N}$ satisfying 
$\Theta_{y,U} (\mathcal{D}_U (y)) = y$;
\item
$\delta_U \tilde{\Theta}_{y,U} (\Tilde{\mathcal{D}}_U (y)) = y$. 
\end{enumerate}
\end{prop}

\begin{proof}
The proof is left to the reader.
\end{proof}

The case where $y$ does not have any prefix in $U$ is treated in the proof of Theorem \ref{theo:morphicdesc}.

The sequence $\tilde{\mathcal{D}}_U (y)$ is called the {\em induced sequence} on $U$ of $y$ and $\mathcal{D}_U (y)$ is called the {\em derived sequence} on $U$ of $y$.
Observe that if $y$ is uniformly recurrent then these sequences are also uniformly recurrent.

When $U=\{u \}$ we set (and we get) $\delta_u = \delta_U$, ${\mathcal{R}}_{y,u} = \tilde{\mathcal{R}}_{y,U}=\mathcal{R}_{y,U}$ and $\mathcal{D}_u (y) = \tilde{\mathcal{D}}_U (y)= \mathcal{D}_U (y)$.

Let $y$ be uniformly recurrent and $u$ a prefix of $y$. 
Then $R_{y,u}$ is finite. 
When $u$ is the first letter of $y$, that is $u=y_0$, then we set $\mathcal{D} (y) = \mathcal{D}_u (y)$ and 
$\mathcal{D}^n (y) = \mathcal{D} (\mathcal{D}^{n-1} (y))$ for all $n\geq 1$.

\section{Bounds for substitutions and derived sequences}
\label{section:bounds}
For the purpose of this paper we need to generalize the results in \cite{Durand:1998} (dealing with primitive substitution sequences) to uniformly recurrent morphic sequences.

\subsection{Derived sequences}
\label{subsec:deriveprim}

Here we recall some known properties of uniformly recurrent sequences and primitive substitutive sequences.

\begin{prop}[\cite{Durand:1998}]
\label{prop:derder}
Let $x$ be a uniformly recurrent sequence and $u$ a non-empty prefix of $x$.
\begin{enumerate}
\item 
The set $\mathcal{R}_{x,u}$ is a code, i.e., $\Theta_{x,u} : R_{x,u}^* \rightarrow \Theta_{x,u} (R_{x,u} ^*)$ is one-to-one. 
\item 
If $u$ and $v$ are two prefixes of $x$ such that $u$ is a prefix of $ v$ then each return word to $v$ belongs to $\Theta_{x,u} (R_{x,u} ^*)$, i.e., it is a concatenation of return words to $u$.
\item 
Let $v$ be a non-empty prefix of $\mathcal{D}_u (x)$ and $w = \Theta_{x,u}(v)u$. Then 
\begin{itemize}
\item 
$w$ is a  prefix $x$,
\item $\Theta_{x,u}\Theta_{\mathcal{D}_u(x),v} = \Theta_{x,w}$ and
\item $\mathcal{D}_v (\mathcal{D}_u (x)) = \mathcal{D}_w(x)$. 
\end{itemize}
\end{enumerate}
\end{prop}

\begin{coro}
\label{coro:Dn}
Let $x\in A^\mathbb{N}$ be uniformly recurrent. 
For all $n\geq 1$ there exists a prefix $u_n$ of $x$ such that 

$$
\mathcal{D}^n (x) = \mathcal{D}_{u_n} (x) \hbox{ and } u_{n+1} = \Theta_{x,u_n} (1)u_n .
$$

\end{coro}

\begin{coro}
\label{coro:theta}
Let $x$ be a uniformly recurrent sequence and $u$ a non-empty prefix of $x$.
If $u$ is a prefix of $v$ and $v$ is a prefix of $x$ then there is a unique morphism $\Theta $ satisfying $\Theta_{x,u} \Theta = \Theta_{x,v}$.
\end{coro}

The next lemma asserts that the derived sequences of purely primitive substitutive sequences are purely primitive substitutive sequences.

\begin{prop}[\cite{Durand:1998}]
\label{prop:derivedsub}
Let $x=\sigma^\infty (a)$ where $\sigma : A^* \to A^*$ is a primitive substitution prolongable on $a$, and, $u$ be
a non-empty prefix of $x$. 
The derived sequence $\mathcal{D}_u (x) $ is the unique
fixed point of the unique substitution, $\sigma_{u}$, satisfying

$$
\Theta_{x,u} \sigma_u = \sigma \Theta_{x,u}.
$$  

Moreover, $\sigma_u$ is prolongable on the letter $1$.
\end{prop}

The substitution $\sigma_u$ is called the {\em return substitution} on $u$.

\begin{lemma}
\cite{Durand:2012}
\label{lemma:returnsubalgo}
Let $x=\sigma^\infty (a)$ where $\sigma : A^* \to A^*$ is a primitive substitution prolongable on $a$.
Let $u$ be a non-empty prefix of $x$.
Then, $\sigma_u$ and $\Theta_{x,u}$ are algorithmically computable. 
\end{lemma}

\subsection{Computable bounds for linear recurrence of substitutive sequences}

In the sequel we give some computable bounds concerning the return substitutions and their linear recurrence constants.
They will be important for the decidability problem we are interested in this paper. 
We will use the notation introduced in Proposition \ref{prop:decomprim}.
In the sequel whenever $M = (m_{i,j})$ is a square matrix, $m^{(n)}_{i,j}$ will denote the entry $(i,j)$ of $M^n$.
We recall that, from Perron's theorem, the dominant eigenvalue $\alpha$ of a primitive matrix $M$ is equal to the spectral radius $\rho (M)$ of $M$, and, the eigenvalue $\alpha$ has an eigenvector with strictly positive entries.
Moreover, $M$ has an eigenvector associated with $\alpha$ having positive entries.

\begin{lemma}
\label{lemma:minmaxperron}
Let $M=(m_{i,j})$ be a matrix with positive entries and $x=(x_i)$ be its unique (right-)eigenvector with positive entries such that $\sum x_i = 1$.
Then, for all $i$,

\begin{equation}
\label{eq:minmaxperron}
\min_j
\frac{ m_{i,j}}{\sum_{k} m_{k,j}}
\leq 
x_i
\leq
\max_j
\frac{ m_{i,j}}{\sum_{k} m_{k,j}} .
\end{equation}
\end{lemma}

\begin{proof}
Perron's theorem implies that (see Proposition 5.8 in \cite{Queffelec:2010})
$$
\lim_{n\to \infty}
\frac{ m^{(n)}_{i,j}}{\sum_{k} m^{(n)}_{k,j}}
=
x_i .
$$

If we suppose that the left inequality in \eqref{eq:minmaxperron} does not hold then we can easily obtain the following contradiction: 

$$
\limsup_{n\to \infty}
\frac{ m^{(n)}_{i,j}}{\sum_{k} m^{(n)}_{k,j}}
>
x_i .
$$

The right inequality can be treated in the same way.
\end{proof}

\begin{lemma}
[\cite{Horn&Johnson:1990}]
\label{lemma:hornjohnson1990}
Let $M =(m_{i,j})$ be a $d\times d $-matrix with non-negative entries.
If $M$ has a positive eigenvector $x$, then for all $n=1,2, \dots$ and all $i=1, \dots , d$ we have

$$
\left[
\frac{\min_{1\leq k \leq d} x_k}{\max_{1\leq k \leq d} x_k} 
\right]
\rho (M)^n
\leq
\sum_{j=1}^d m_{ij}^{(n)} 
\leq 
\left[
\frac{\max_{1\leq k \leq d} x_k}{\min_{1\leq k\leq d} x_k} 
\right]
\rho (M)^n.
$$
\end{lemma}

\begin{lemma}
\label{lemma:matrix}
Let $M= (m_{i,j})$ be a $d\times d$-matrix and $M'= (m'_{k,l})$ be a $d'\times d'$-matrix, both being primitive.
Then, there exists a computable constant $Q$ such that for all $i,j,k,l,n$ one gets 
$
(1/Q) m'^{(n)}_{i,j}
\leq
m^{(n)}_{i,j} 
\leq 
Qm'^{(n)}_{i,j} .
$
\end{lemma}

\begin{proof}
This comes from Perron's theorem, Lemma \ref{lemma:minmaxperron} and Lemma \ref{lemma:hornjohnson1990}.
\end{proof}

We recall that for each endomorphism $\sigma$ we have, for all $n$, $|\sigma^n | = \max_j \sum_i m_{ij}^{(n)} $ and  
$\langle \sigma^n \rangle = \min_j \sum_i m_{ij}^{(n)} $.

\begin{lemma}
\label{lemma:maxmin}
Let $\sigma $ be an endomorphism and $\alpha >1$ such that all letters of $A$ has a growth type $(0 , \alpha )$. 
There exists two computable constants $P_\sigma $ and $Q_\sigma$ such that, for all $k$, the following holds:

\begin{equation}
\label{eq:prelim}
(1/P_\sigma ) \alpha^k \leq \langle \sigma^k \rangle \leq |\sigma^k | \leq P_\sigma \alpha^k \hbox{ and }
\end{equation}
\begin{equation}
\label{eq:constantmaxmin}
|\sigma^k | \leq Q_\sigma \langle \sigma^k \rangle .
\end{equation}

\end{lemma}
\begin{proof}
It suffices to prove \eqref{eq:prelim}.
Let $M= (m_{i,j})$ be the incidence matrix of $\sigma$.

When $\sigma $ is primitive,
we conclude with Perron's theorem, Lemma \ref{lemma:hornjohnson1990} and Lemma \ref{lemma:minmaxperron}.

Let consider the non-primitive case.
For all positive integer $l$, consider the following statement.

\begin{center}
$S(l)$
If $M$ has $l$ primitive components, then there exists a computable constant $P_\sigma$ such that, for all $k$, \eqref{eq:prelim} holds.
\end{center}

It suffices to prove $S(l)$ for all $l$.
We will proceed by induction.
Notice that $S(1)$ is true because it corresponds to the primitive case.
We suppose $S(l)$ is true. 
Let us prove $S(l+1)$ also holds.

Let $\{ A_k | 1\leq k\leq l+1 \}$ be a partition of $A$ that is organized as in Proposition \ref{prop:decomprim}. 
Consider the primitive matrices $M'_k = (m_{i,j})_{i\in A_k,j\in A_k}$.  
With the notation of Proposition \ref{prop:decomprim} we get ${M'}_k^{r_\sigma} = M_k$.
Let $B = \cup_{2\leq k\leq l} A_k$ and $N = (m_{i,j})_{i\in B,j\in B}$.
Remark that $\sigma_{/B}$, the restriction of $\sigma$ to $B^*$, is an endomorphism whose incidence matrix is $N$.
>From the hypothesis, the spectral radius of $N$ is $\alpha$.
Observe that $N$ has $l$ primitive components and thus, from the induction hypothesis, there exists a computable constant $P$ such that \eqref{eq:prelim} holds.

Let $\sigma_1$ be an endomorphism whose incidence matrix is $M_1$.
As $M_1$ is primitive, there exists a computable constant $P_1$ such that \eqref{eq:prelim} holds for $\sigma_1$.

Let $a\in A_1$.
We set $\sigma (a) = c_0 c_1 \cdots c_{|\sigma (a)|-1}$ where the $c_i$'s belong to $A$.
We have:

$$
|\sigma^k (a)| = |\sigma_1^k (a)| + \sum_{0\leq i\leq |\sigma (a)|-1} |\sigma_{/B}^{k-1} (c_i)| .
$$

This achieves the proof.
\end{proof}

A sequence $x$ is said to be {\em linearly recurrent} (for the constant $K$) if it is uniformly recurrent and the difference between two successive occurrences of any word $u$ is bounded by $K|u|$.

In the sequel $R_\sigma$ will denote the maximal difference between two successive occurrences of a word of length $2$ in any word on $\mathcal{L} ( \sigma )$.

\begin{lemma}
\label{lemma:boundRsigma}
If $\sigma $ is a primitive substitution defined on an alphabet with $d$ letters and prolongable on the letter $a$, then $R_\sigma$ is bounded by $2|\sigma^{2d^2}|$.
\end{lemma}

\begin{proof}
It is easy to check that for all $k\geq d^2$, any word $\sigma^k (a)$ has an occurrence of all words of length 2 appearing in $\sigma^\infty (a)$.
Therefore, from Lemma \ref{lemme:horn}, $R_\sigma$ is bounded by $2|\sigma^{2d^2}|$.
\end{proof}

\begin{prop}[\cite{Durand:1998}]
\label{prop:sublinrec}
Let $x=\sigma^\infty (a)$, where $\sigma$ is a primitive substitution prolongable on $a$. 
Then $x$ is linearly recurrent for the constant $K_\sigma = Q_\sigma R_\sigma |\sigma|$.
\end{prop}

Let $\sigma $ be a substitution. 
>From Proposition \ref{prop:decomprim} there always exists a sub-alphabet $B\subset A$ such that $\sigma (B)$ is a subset of $B^*$ and the incidence matrix of $\sigma_{/B}$ is primitive.
We say that $\sigma_{/B}$ is a {\em primitive sub-morphism} of $\sigma$.
Observe that $\sigma $ has finitely many primitive sub-morphisms.

\begin{coro}
\label{coro:Rsigma}
Let $x=\phi (\sigma^\infty (a))$ be a uniformly recurrent sequence where $\phi$ is a coding and $\sigma :A^*\to A^*$ is a substitution prolongable on $a$.
Then, $x$ is linearly recurrent for the constant $K$ where $K$ is the infimum of the $Q_\tau R_\tau |\tau |$, $\tau$ running over all primitive sub-morphisms of $\sigma$.
\end{coro}

\begin{proof}
Let $\tau : B^* \to B^*$ be a primitive sub-morphism of $\sigma$. 
Taking a power of $\tau$ if needed, we can suppose $\sigma_{/B}$ is prolongable on some letter $b$.
Notice that, $\phi (\sigma^\infty (b))$ being uniformly recurrent, we get $\mathcal{L} (\phi (\sigma^\infty (b))) = \mathcal{L} (x)$.
We conclude with Proposition \ref{prop:sublinrec}.
\end{proof}

\begin{theo}[\cite{Durand&Host&Skau:1999}]
\label{theo:encad}
Let $x$ be a non-periodic linearly recurrent sequence for the constant $K$. 
For all words $u$ of $\mathcal{L} (x)$, 
\begin{enumerate}
\item
for all $n$, every word of $\mathcal{L}_n (x)$ has an occurrence in each word of $\mathcal{L}_{K(n+1)} (x)$,
\item 
\label{theo:item:1}
for all words $v$ belonging to $\mathcal{R}_{x,u}$, $\frac{1}{K}| u| \leq | v| \leq K  | u|$, 
\item 
\label{theo:item:2}
$\# (\mathcal{R}_{x,u})\leq 4K^3$,
\item
$p_x(n) \leq (K+1)n$ for all $n$.
\end{enumerate}
\end{theo}

>From this theorem and the definition of $\sigma_u$ we deduce the following corollary.

\begin{coro}
For all primitive substitutions $\sigma$, prolongable on $a$ such that $\sigma^\infty (a)$ is non-periodic, and all $u$ such that $u$ is a prefix of $\sigma (u)$ we have

$$
| \sigma_u | \leq |\sigma |K_\sigma^2.
$$

\end{coro}

\subsection{Preimages of codings giving uniformly recurrent morphic sequences}

In this section $\phi$ is a coding and $\sigma :A^*\to A^*$ is a substitution prolongable on $a$ such that all its letters have a growth type. From Proposition \ref{prop:croissance} there is no restriction to make such an assumption on the existence of the growth types.
In addition we suppose the growth rate of $\sigma $ is $\alpha$ and that $x=\phi (y)$ is uniformly recurrent and non-ultimately periodic where $y=\sigma^\infty (a)$.
We also suppose that all letters of $A$ occur in $y$.
Observe that $\sigma$ being a substitution (and thus being growing), $\alpha $ is strictly greater than $1$.
>From Corollary \ref{coro:Rsigma}, $x$ is linearly recurrent for some computable constant $K$.

The following two lemmas were proven in \cite{Durand:2011}.

\begin{lemma}
[Lemma 11 in \cite{Durand:2011}]
\label{lemma:concatenationrw}
Let $t$ be a non-periodic linearly recurrent sequence (with constant $L$).
Then, for all $u\in \mathcal{L}(t)$ and all $l\in \mathbb{N}$, 

$$
\# \cup_{0\leq n \leq l} \mathcal{R}_{t,u}^* \cap \mathcal{L}_n(t) \leq \left( 1+ 4L^3 \right)^{\frac{lL}{|u|}} .
$$
\end{lemma}

\begin{lemma}
[Lemma 15 in \cite{Durand:2011}]
\label{lemma:samegrate}
All letters of $A$ have the same growth rate (w.r.t $\sigma$).
\end{lemma}

Reading carefully the proof of this lemma in \cite{Durand:2011} we can slightly improve it to obtain the following lemma.

\begin{lemma}
\label{lemma:samegtype}
All letters of $A$ have the same growth type $(0,\alpha)$ (w.r.t $\sigma$).
\end{lemma}

\begin{proof}
>From Lemma \ref{lemma:samegrate} we know that all letters have the same growth rate.
We suppose that  the growth type of $\sigma $ is $(d,\alpha )$ with $d\geq 1$.

There exists a letter whose growth type is $(0,\alpha)$ (Proposition \ref{prop:croissance}).
Hence, there necessarily exists a non-empty word $bwc$, appearing infinitely many times in $y$, where $w$ has growth type $(0, \alpha )$, and, $b$ and $c$ are letters with $d(b)\geq 1$ and $d (c) \geq 1$.
 
Once we observe that for all letters $\gamma$ with growth type $(f,\alpha )$, with $f\geq 1$, there exist letters $\gamma'$ and $\gamma''$ having growth type $(f ,\alpha )$, such that $\sigma (\gamma ) = u'\gamma'u=v\gamma''v'$ where $u$ and $v$ (being possibly the empty word) have growth types strictly less than $(f,\alpha )$, then,
by a recurrence starting with $bwc$, it is easy to prove that there exist two sequences of letters $(b_n)_n$ and $(c_n)_n$, and, two sequences of words $(u_n)_n$ and $(v_n)_n$ such that for all $n\in \mathbb{N}$:

\begin{enumerate}
\item
$b_n u_n \sigma (u_{n-1}) \cdots \sigma^{n-1} (u_1) \sigma^{n} (w) \sigma^{n-1} (v_1) \cdots \sigma (v_{n-1}) v_n c_n $
appears infinitely many times in $x$;
\item
$b_n$ has growth type $(f , \alpha)$;
\item
$c_n$ has growth type $(f , \alpha)$;
\item
$u_n$ has growth type $(d_n,\alpha)$ with $d_n < f$, $|u_n|< \max_{c'\in A} |\sigma (c')|$;
\item
$v_n$ has growth type $(e_n,\alpha)$ with $e_n < f$, $|v_n|< \max_{c'\in A} |\sigma (c')|$.
\end{enumerate}

Consequently, for all $n$ and $k$ the word $W(n,k) = \sigma^k (b_n)U(n,k)\sigma^k (c_n)$, where

$$
\begin{array}{l}
U(n,k) = \sigma^k(u_n) \cdots \sigma^{k+n-1} (u_1) \sigma^{k+n} (w) \sigma^{k+n-1} (v_1) \cdots  \sigma^k (v_n)  ,
\end{array}
$$

appears infinitely many times in $x$. 
There exist a strictly increasing sequence $(m_i)$ and two letters $b''$ and $c''$ such that 

\begin{equation}
\label{proof:aprime}
b_{m_i} = b'' \ \ {\rm and} \ \  c_{m_i} = c'' \ \ \hbox{\rm for all} \ \ i.
\end{equation}

>From Theorem \ref{theo:encad} and Lemma  \ref{lemma:concatenationrw} there exists some constant $L$ such that:
\begin{enumerate}
\item[(P1)]
$x$ is linearly recurrent for the constant $L$;
\item[(P2)]
for all $u\in \mathcal{L} (x)$  and $w\in \mathcal{R}_{x,u}\cap \mathcal{L} (x)$ we have 
$ \frac{|u|}{L} \leq |w|\leq L|u|$;
\item [(P3)]
for all $u\in \mathcal{L} (x)$ we have $\# \mathcal{R}_{x,u}\cap \mathcal{L} (x)\leq L$;
\item [(P4)]
for all $u\in \mathcal{L} (x)$ and all $l\in \mathbb{N}$
$$
\# \cup_{0\leq n \leq l} \left( \mathcal{R}_{x,u}^* \cap \mathcal{L}_n(x)\right) \leq \left( 1+ L \right)^{\frac{lL}{|u|}} .
$$
\item [(P5)]
for all $n$ and $k$
\begin{align*}
\label{proof:aprimeprime}
|U(n,k)|   & \leq L\left( \sum_{i=0}^n (k+i)^{f-1} \alpha^{k+i} \right) ,\\
\frac{1}{L} k^{f} \alpha^{k}  & \leq |\sigma^k (b'')| \leq L k^{f}  \alpha^{k} \hbox{ and }\\
\frac{1}{L} k^{f}\alpha^{k}   & \leq |\sigma^k (c'')| \leq L k^{f}  \alpha^{k} .
\end{align*}
\end{enumerate}

>From \eqref{proof:aprime} and the previous inequalities, there exists $k$ such that $|\sigma^k (c'')|\leq |\sigma^k (b'')|$ (the other case can be treated in the same way), $2(1+L)\leq |\sigma^k (c'')|$ and 

\begin{align}
|U(m_i,k)| \leq |\sigma^k (c'')| \hbox{ for all } 1\leq i\leq \left( 1+ L \right)^{2(L+3)(L+1)L} +1 .
\end{align}

>From (P1) and (P2), for all $j$, all words in $\mathcal{L}_j (x)$ appear in all words of $\mathcal{L} _{(L+1)j} (x)$.
Let $u$ be a prefix of $\phi (\sigma^k (c''))$ such that  

$$
\frac{|\sigma^k (c'')|}{L+1} -1 \leq |u| \leq  \frac{|\sigma^k (b'')|}{L+1} .
$$

Then $u$ is non-empty and occurs in $\phi (\sigma^k (b''))$.
We can decompose $\phi (\sigma^k (b''))$ and $\phi (\sigma^k (c''))$ in such a way that $\phi (\sigma^k (b'')) = G_puG_s$ and $\phi (\sigma^k (c'')) = uH_s$ with

$$
|G_s|\leq (L+1)|u| .
$$

Observe that $u$, $b''$, $c''$ and $G_s$ do not depend on $i$.
Moreover, for all $i$ belonging to $[1 , \left( 1+ L \right)^{2(L+3)(L+1)L} +1]$, the word $uG_s \phi (U(m_i,k))$ belongs to $\mathcal{L} (x)$, is a concatenation of return words to $u$ and satisfies

$$
|uG_s \phi (U(m_i,k))| \leq (L+3)|\sigma^k (c'')|  
. $$

Moreover, from (P4) we get

\begin{align}
\label{ineg:rwcontradict}
\# \bigcup_{n=0}^{(L+3)|\sigma^k (c'')|} \mathcal{R}_{x,u}^* \cap \mathcal{L}_n(x) 
& \leq 
\left( 1+ L \right)^{\frac{(L+3)|\sigma^k (c'')|L}{|u|}}
 \leq 
\left( 1+ L \right)^{2(L+3)(L+1)L} .
\end{align}

But observe that $(|U(m_i,k)|)_{0\leq i \leq 1 + ( 1+ L )^{2(L+3)(L+1)L}}$ being strictly increasing, the words 
$uG_s \phi (U(m_i,k)) $, $0\leq i \leq 1 + ( 1+ L )^{2(L+3)(L+1)L}$ are all distinct (and belong to $\cup_{0\leq n \leq (L+3)|\sigma^k (c'')|  } \mathcal{R}_{x,u}^* \cap \mathcal{L}_n(x)$). 
This is in contradiction with \eqref{ineg:rwcontradict}.
\end{proof}

We can now conclude with the number of preimages of $\phi$.

\begin{prop}
\label{prop:boundedpreimages}
For all $u\in \mathcal{L} (x)$, we have: $\# \phi^{-1} (\{ u\})\cap \mathcal{L} (y) \leq p_\sigma (K+1) |\sigma |Q_\sigma (K+1)^2$. 
\end{prop}

\begin{proof}
Let $u\in \mathcal{L} (x)$ and $n$ be such that $\langle \sigma^n \rangle \leq |u| \leq \langle \sigma^{n+1}\rangle$.
>From Theorem \ref{theo:encad},
all $v\in \phi^{-1} (\{ u\})\cap \mathcal{L} (y)$ occur in some $\sigma^{n+1} (a_1 \cdots a_{K +1})$ and in such a word occurs at most 
$| \sigma^{n+1} (a_1 \cdots a_{K+1})|/(|u|/K )$ occurrences of elements belonging to $\phi^{-1} (\{ u\})$. 
Hence, using Lemma \ref{lemma:samegtype} and Lemma \ref{lemma:maxmin}, one gets

$$
|\phi^{-1} (\{ u\})\cap \mathcal{L} (y)| 
\leq p_\sigma (K+1) \frac{(K+1)|\sigma^{n+1}|}{|u|/ K } 
\leq p_\sigma (K+1) |\sigma |Q_\sigma (K+1)^2 .
$$
\end{proof}

\subsection{Derived sequences of uniformly recurrent substitutive sequences}
\label{subsec:derivedsequnifrec}



We keep the assumptions on $x$, $y$, $\sigma$ and $\phi$.
Taking a power of $\sigma$ if needed, we can suppose $\langle\sigma \rangle \geq (K+1)^2$. 

Let $u$ be a non-empty prefix of $y$, $v=\phi (u)$ and $U=\phi^{-1} (\{ v \})$.
>From Theorem \ref{theo:encad},  for each word $w$ such that $|w|\geq (|u|+1)/K$, the word $\sigma (w)$ has an occurrence of an element of $U$.
Thus, for such $w$ we set 

$$
\sigma (w) = p(w) m(w) s(w) 
$$

where $m(w)$ belongs to $U$ and $p(w)m(w)$ has a unique occurrence of an element of $U$.
Observe that $p(u)$ is the empty word and $m(u)=u$.
Now take $(w,u') \in \tilde{\mathcal{R}}_{y,U}$. 
Then 

$$
\sigma (wu') = p(w) m(w) s(w) p(u') m(u') s(u') 
$$

and there exists a unique finite sequence $(w_1,u_1), \dots , (w_k, u_k)$ in $\tilde{\mathcal{R}}_{y,U}$ such that

\begin{enumerate}
\item
$w_1 \cdots w_k = m(w) s(w) p(u')$ and 
\item
$w_1 w_2 \cdots w_i u_i$ is a prefix of $m(w) s(w) p(u') m(u')$ for all $i$.
\end{enumerate}

We set 

$$
\tilde{\sigma}_U (w,u') = (w_1,u_1) \dots  (w_k, u_k) .
$$

This defines an endomorphism of $\tilde{\mathcal{R}}_{y,U}^*$.
Then, we define $\sigma_U$ to be the unique endomorphism of $\tilde{R}_{y,U}^*$
satisfying 

\begin{align}
\label{def:sigmaU}
\tilde{\sigma}_U \tilde{\Theta}_{y,U} = \tilde{\Theta}_{y,U} \sigma_U .
\end{align}

Observe that the morphism $\tilde{\Theta}_{y,U}$ being a bijection from $\tilde{R}_{y,U}$ onto $\tilde{\mathcal{R}}_{y,U}$, we get $|\sigma_U (i)| = |\tilde{\sigma}_U (w,u')|$ 
for all $i\in \tilde{R}_{y,U}$ and $(w,u') \in \tilde{\mathcal{R}}_{y,U}$ such that $\tilde{\Theta}_{y,U} (i) = (w,u')$.

We recall that when $y$ is a primitive substitutive sequence and $U$ is a singleton $\{ u \}$ then $\sigma_U = \sigma_u$ and it is the unique endomorphism satisfying $\sigma \Theta_{y,u} = \Theta_{y,u} \sigma_u $ (Proposition \ref{prop:derivedsub}).
Unfortunately $\sigma_U$ does not satisfy such an equality (involving $\sigma$) in general. 
But we are not far from such a situation. 
This is explained by the following lemma which can be deduced from the above definitions.

\begin{lemma}
\label{lemma:defsigmaU}
Let $u$ be a non-empty prefix of $y$ and $U$ be the set $\phi^{-1} (\{ \phi (u) \})$.
Then, for all letters $(w_1,u_1), \dots , (w_k, u_k)$ in $\tilde{\mathcal{R}}_{y,U}$ such that $w_1 w_2 \cdots w_k u_k$ belongs to $\mathcal{L} (y)$ and $w_1 w_2 \cdots w_i u_i$ is a prefix of $w_1 w_2 \cdots w_k u_k$, for all $i\leq k$, we have

$$
p(w_1)
\delta_U \left( 
\tilde{\sigma}_U (w_1 , u_1) \cdots \tilde{\sigma}_U (w_k , u_k) 
\right)
m(w_k)s(w_k) 
=
\sigma (w_1\dots w_ku_k) .
$$

Moreover, for the unique element $(w,u)\in \tilde{R}_{y,U}$ such that $wu$ is a prefix of $y$, $\tilde{\sigma}_U^\infty (w,u)$ is well-defined, 

$$
\tilde{\sigma}_U^\infty (w,u) = \Delta_U (y) \hbox{ and } \delta_U (\tilde{\sigma}_U^\infty (w,u) ) = y.
$$
\end{lemma}

\begin{prop}
\label{prop:sigmaUfinite}
For all prefixes $u$ of $y$, with $U=\phi^{-1} (\{ \phi (u) \})$, we get:

\begin{enumerate}
\item
\label{enum:1}
The endomorphism $\sigma_U$ is a prolongable substitution on the letter $1$;
\item
\label{enum:2}
$\sigma_U^\infty (1) = \tilde{\mathcal{D}}_U (y)$;
\item
\label{enum:3}
$\phi \delta_U \tilde{\Theta}_{y,U} \sigma_U^\infty (1) = x$;
\item
\label{enum:4}
There exists a unique coding $\psi_u : \tilde{R}_{y,U}^* \to \tilde{R}_{x,u}^*$ satisfying

$$
\phi \delta_U \tilde{\Theta}_{y,U} = \Theta_u \psi_u ;
$$
\item
\label{enum:5}
$\psi_u (\tilde{\mathcal{D}}_U (y)) = \mathcal{D}_u (x)$.
\end{enumerate}
\end{prop}

\begin{proof}
Statement \eqref{enum:1} follows from Lemma \ref{lemma:defsigmaU} and the definition of $\sigma_U$ in \eqref{def:sigmaU}.
Statement \eqref{enum:3} is due to Lemma \ref{lemma:defsigmaU} and Proposition \ref{prop:def-DU}.
Statement \eqref{enum:2} also needs \eqref{def:sigmaU} to be proven.
Statement \eqref{enum:4} is easy to establish and Statement \eqref{enum:5} follows from \eqref{enum:2}, \eqref{enum:3}, \eqref{enum:4} and Proposition \ref{prop:def-DU}. 
\end{proof}

The following theorem extends the main result in \cite{Durand:1998} and proves Theorem \ref{theo:maindurand1998} (see Statement \eqref{enum:der4} below) for uniformly recurrent substitutive sequences, the necessary condition being already proven in \cite{Durand:1998}.

\begin{theo}
\label{theo:derbound}
We get:

\begin{enumerate}
\item
\label{enum:der1}
For all prefixes $u$ of $y$, with $U=\phi^{-1} (\{ \phi (u) \})$, $\#  \tilde{R}_{y,U} \leq K_1 = 4K^3 p_\sigma (K+1) |\sigma |Q_\sigma (K+1)^2$;
\item
\label{enum:der2}
$\# \{ \sigma_U  | U= \phi^{-1} (\{ \phi (u) \}), u \hbox{ is a prefix of } x \} \leq K_1^{K_1 K_2}$, with $K_2 = |\sigma |(K+1)K$;
\item
\label{enum:der3}
$\# \{ \psi_u  | u \hbox{ prefix of } x \}\leq K_1^2$;
\item
\label{enum:der4}
The cardinality of the set of derived sequences $\{ \mathcal{D}_u (x) | u \hbox{ is a prefix of } x \}$ is bounded by $K_1^{K_1K_2+2}$.
\end{enumerate}
\end{theo}

\begin{proof}
Let $u$ be a non-empty prefix of $y$.
We set $U=\phi^{-1} (\{ \phi ( u ) \})$.
Let us first show that $\tilde{R}_{y,U}$ is bounded independently of $U$.
We have to bound the number of words $wu'$ where $(w,u' )$ belongs to $\tilde{\mathcal{R}}_{y,U}$.
The sequence $x$ being linearly recurrent for the constant $K$,
from Proposition \ref{prop:boundedpreimages} and Proposition \ref{theo:encad}

$$
\#  \tilde{R}_{y,U} \leq K_1 = 4K^3 p_\sigma (K+1) |\sigma |Q_\sigma (K+1)^2 .
$$

Now let us bound the length of $\sigma_U$.
Let $i \in \tilde{R}_{y,U}$ and $(w,u') = \tilde{\Theta}_{y,U} (i)$. Thus, $(w,u')$ belongs to $\tilde{\mathcal{R}}_{y,U}$.
>From the definition of $\sigma_U$ and Lemma \ref{lemma:defsigmaU} we get

$$
|\sigma_U (i)| = |\tilde{\sigma}_U (w,u')| \leq \frac{|\sigma (wu') |}{\frac{|u|}{K}} \leq K_2 = |\sigma | K(K+1) . 
$$

Hence 

$$
\# \{ \sigma_U  | U= \phi^{-1} (\{ \phi (u) \}), u \hbox{ prefix of } x \} \leq (K_1+1)^{K_1 K_2} .
$$

Let us show that $\# \{ \psi_u  | u \hbox{ prefix of } x \}$ is bounded.
The morphism $\psi_u$ is a coding that necessarily verifies $\psi_u (\tilde{R}_{y,U})= \tilde{R}_{x,u}$.
Thus, we obtain \eqref{enum:der3} using \eqref{enum:der1}.

The last statement follows from the two previous statements and Proposition \ref{prop:sigmaUfinite}. 
This achieves the proof.
\end{proof}

\section{Decidability of uniform recurrence of morphic sequences}
\label{section:decid}

In the sequel $\sigma : A^*\to A^*$ is an endomorphism prolongable on $a$, $\phi : A^* \to B^*$ is a morphism, $y=\sigma^\infty  (a)$ and $x=\phi (y)$ is a sequence of  $B^\mathbb{N}$.
We have to find an algorithm telling if $x$ is uniformly recurrent or not.

>From Proposition \ref{prop:croissance} it is decidable to find an integer $p$ such that all letters of $A$ has a growth type w.r.t. $\sigma^p$.
Hence, taking such a power if needed, we suppose $\sigma $ has this property. 
In this way we can freely use Theorem \ref{theo:derbound}.

In \cite{Durand:preprint2011} it is shown that the ultimate periodicity of $x$ is decidable. 
Hence we suppose $x$ is non-periodic.

As explained in Section 3.2 of \cite{Durand:preprint2011} we can always suppose that all letters of $A$ occur in $y$.
When it is not the case, $\sigma$ and $\phi$ can be algorithmically changed to fulfill this assumption. 
Hence we make such an assumption. 
An other reduction of the problem is given by the following result. 

\begin{theo}
\cite{Honkala:2009,Durand:preprint2011}
There exists an algorithm  computing a coding $\varphi$ and a non-erasing endomorphism $\tau$, prolongable on $a'$, such that $x=\varphi (z)$ where $z=\tau^\infty (a')$.
\end{theo}

Consequently we can consider that $\phi $ is a coding and $\sigma $ is non-erasing.

>From Proposition \ref{prop:decomprim} it is decidable whether $\sigma$ is growing or not.
Moreover we have the following lemma.

\begin{lemma}
\label{lemma:pansiot}
\cite[Th\'eor\`eme 4.1]{Pansiot:1984}
The endomorphism $\sigma$ satisfies exactly one the following two statements.
\begin{enumerate}
\item
\label{lemma:enum:pansiot2}
The length of words (occurring in $\sigma^\infty (a)$) consisting of non-growing letters is bounded.
\item
\label{lemma:enum:pansiot3}
There exists a growing letter $b\in A$, occurring in $\sigma^\infty (a)$, such that $\sigma (b) =  vbu$ (or $ubv$) with $u\in C^*\setminus \{ \epsilon \}$ where $C$ is the set of non-growing letters.
\end{enumerate}
Moreover, in the situation \eqref{lemma:enum:pansiot2} the sequence $\sigma^\infty (a)$ can be algorithmically defined as a  substitutive sequence.
\end{lemma}

Suppose $\sigma $ is not growing.
>From Lemma \ref{lemma:pansiot} we can also suppose $\sigma $ satisfies the Statement \eqref{lemma:enum:pansiot3} of this lemma.
Consequently it is classical to check there exists a word $w$, that can be algorithmically determined, such that $w^n$ belongs to $\mathcal{L} (y)$ for all $n$. 
Thus, to be uniformly recurrent, $x$ should be periodic, more precisely we should have $x=\phi (w)^\infty$. 
This can be verified using the main result in \cite{Durand:preprint2011} or \cite{Mitrofanov:preprint2011a}, or, to use the following lemma.

\begin{lemma}
\label{lemma:finalcheck}
$x=\phi (w)^\infty$ if and only if

\begin{enumerate}
\item
for all words $B = b_1 b_2 \cdots b_{2|w|}\in \mathcal{L} (y)$, where the $b_i$'s are letters, $\phi (B) = s_{B}\phi (w)^{r_{B}}p_{B}$ where $r_B$ is a positive integer,  $s_{B}$ is a suffix of $\phi (w)$ and $p_{B}$ a prefix of $\phi (w)$;
\item
for all $BB' \in \mathcal{L} (y)$, where $B$ and $B'$ are words of length $2|w|$, $p_{B}s_{B'}=\phi (w)$;
\item
when $B=\phi (y_0y_1 \cdots y_{2|w|-1})$, $s_B$ is the empty word.  
\end{enumerate}
\end{lemma}

\begin{proof}
The proof is left to the reader.
\end{proof}

As the set of words of length $2|w|$ is algorithmically computable, it is decidable whether $x$ is uniformly recurrent (in fact periodic) when $\sigma$ is not growing.

\medskip

Consequently we suppose $\sigma$ is a substitution.

Let $K$ be the infimum of the $Q_\tau R_\tau |\tau |$, $\tau$ running over all primitive sub-morphisms of $\sigma$.
If $x$ is uniformly recurrent, then, from Corollary \ref{coro:Rsigma}, $x$ should be linearly recurrent for the constant $K$. 
In this case, $K$ being fixed with respect to the linear recurrence of $x$, we can take a power of $\sigma$ instead of $\sigma$: this will change neither $x$ nor $K$. 
Consequently, we can suppose that $\langle \sigma \rangle \geq (K+1)^2$ and the constants of Section \ref{subsec:derivedsequnifrec} can be freely used whenever we suppose $x$ uniformly recurrent.

The algorithm we now present is based on the fact (proven incidently below) that $x$ is uniformly recurrent if and only if in the set $\{ \mathcal{D}^n (x) | 1\leq n\leq 1+ K_1^{K_1K_2 +2} \}$ there are two identical sequences, where $K_1$ and $K_2$ are the computable constants given by Theorem \ref{theo:derbound}.

Let $u_1 = y_0$.
We start the algorithm computing $\sigma_{U_1}$ and $\psi_{u_1}$.
There are several intermediate objects that have to be determined. 
Let us remind the order in which we should proceed (see Section \ref{subsec:derivedsequnifrec}):

\begin{enumerate}
\item
\label{decid1}
$U_1 = \phi^{-1} (\{u_1\} )$,
\item
\label{decid2}
$\tilde{\mathcal{R}}_{y,U_1}$, $\mathcal{R}_{y,U_1}$, $\tilde{R}_{y,U_1}$ and $R_{y,U_1}$,
\item 
\label{decid3}
$\tilde{\Theta}_{y,U_1}$,
\item
\label{decid4}
$\tilde{\sigma}_{U_1}$ and $\sigma_{U_1}$, 
\item
\label{decid5}
$\delta_{U_1}$, $\Theta_{u_1}$ and $\psi_{u_1}$.
\end{enumerate}

Let us explain how to compute $\tilde{\mathcal{R}}_{y,U_1}$, $\mathcal{R}_{y,U_1}$, $\tilde{R}_{y,U_1}$ and $R_{y,U_1}$.
It suffices to compute $\tilde{\mathcal{R}}_{y,U_1}$.
In fact we will directly compute $\tilde{\sigma}_{U_1}$.
We will need the following lemmas.

\begin{lemma}
All words in $\mathcal{L}_2 (x)$ have an occurrence in $\sigma^{\# A^2} (x_0x_1)$. 
\end{lemma}

\begin{proof}
The proof is left to the reader.
\end{proof}

As a consequence we obtain the following lemma.

\begin{lemma}
\label{lemma:returniterate}
For all $n$, all words in $\mathcal{L}_n (x)$ have an occurrence in $\sigma^{n+\# A^2} (x_0x_1)$. 
\end{lemma}

\begin{proof}
The proof is left to the reader.
\end{proof}

\begin{lemma}
\label{lemma:UER}
Suppose $x$ is uniformly recurrent.
For all $(w',u') \in \tilde{\mathcal{R}}_{y,U_1}$, the word $w'u'$ has an occurrence in the word
$\sigma^{K+1 +\#A^2} (wu_1)$.
\end{lemma}

\begin{proof}
>From Lemma \ref{lemma:returniterate}, all words of length $K+1$ have an occurrence in the word $\sigma^{K+1+\# A^2} (wu_1)$.
We achieve the proof observing that, $x$ being uniformly recurrent, for all $(w',u') \in \tilde{\mathcal{R}}_{y,U_1}$ we get $|w'u'|\leq K+1$. 
\end{proof}

In $y_{[0,K]}$ there should be at least two occurrences of $u_1$: $0$ and some other. 
If this is not the case, then $x$ is not uniformly recurrent.
Let $j$ be the second occurrence of $u_1$ is $y$. 
Then, $(w_1,u_1) =(y_{[0,j-1]}, u_1)$ is an element of $\tilde{\mathcal{R}}_{y,U_1}$.
(In what follows we use the notations of Section \ref{subsec:derivedsequnifrec}.)
Below, all the $(w_i,u_i)$ will be elements of $\tilde{\mathcal{R}}_{y,U_1}$.
We have:

\begin{itemize}
\item
$\tilde{\sigma}_U (w_1,u_1 )= (w_1 , u_1)\cdots (w_{r_1} , u_{r_1} )$ where 
\item
$\sigma (w_1u_1) = w_1 \dots w_{r_1} m(u_{r_1})s(u_{r_1})$ and is a prefix of $y$, and,
\item
$m(u_{r_1})s(u_{r_1})$ is a suffix $\sigma (u_1)$.
\end{itemize}

Observe that $r_1$ is necessarily greater than or equal to $2$.
All the $w_i$ should have a length less than $K$, otherwise $x$ is not uniformly recurrent. 
We continue with $(w_2 , u_2)$: 

\begin{itemize}
\item
$\tilde{\sigma}_U (w_2,u_2 )= (w_{r_{1}+1} , u_{r_{1}+1})\cdots (w_{r_2} , u_{r_2} )$ where $r_2\geq r_1+1$ and 
\item
$\sigma (w_1 w_2 u_2) = w_1 \dots w_{r_2} m(u_{r_2})s(u_{r_2})$ and is a prefix of $y$, and,
\item
$m(u_{r_2})s(u_{r_2})$ is a suffix $\sigma (u_2)$.
\end{itemize}

All the $w_i$ should have a length less than $K$, otherwise $x$ is not uniformly recurrent. 

Once we proceeded $I$ steps we get: for $1\leq i\leq I$,

\begin{itemize}
\item
$\tilde{\sigma}_U (w_i,u_i )= (w_{r_{i-1}+1} , u_{r_{i-1}+1})\cdots (w_{r_{i}} , u_{r_{i}} )$ where $r_{i}\geq r_{i-1}+1$ and 
\item
$\sigma (w_1 w_2\cdots w_i u_i) = w_1 \dots w_{r_{i}} m(u_{r_{i}})s(u_{r_{i}})$ and is a prefix of $y$, and,
\item
$m(u_{r_i})s(u_{r_i})$ is a suffix $\sigma (u_i)$.
\end{itemize}

We stop when $I$ satisfies $|w_1 \dots w_{r_{I}}|  \geq |\sigma^{K+1 +\#A^2} (wu_1)|$.
Indeed, from Lemma \ref{lemma:UER}, if $x$ is uniformly recurrent then we can find all the elements of $\tilde{\mathcal{R}}_{y,U_1}$ looking at $\sigma (w_1 w_2\cdots w_I u_I)$.
If we find more than $K_1$ elements then $x$ is not uniformly recurrent. 
Then we check that for the step $I+1$ no new elements of $\tilde{\mathcal{R}}_{y,U_1}$ appear. 
If some appears, then $x$ is not uniformly recurrent. 
Otherwise, no new such element would appear in any steps $i$ greater than $I$. 

Then, we continue computing $\mathcal{R}_{y,U_1}$, $\tilde{R}_{y,U_1}$, $R_{y,U_1}$ and $\tilde{\Theta}_{y,U_1}$, and,
$\sigma_{U_1}$, $\delta_{U_1}$, $\Theta_{u_1}$ and $\psi_{u_1}$ using their definition.

We recall that when $x$ is uniformly recurrent then, from Corollary \ref{coro:Dn}, for all $n$ we get 

\begin{equation}
\label{eq:derivedUR}
\mathcal{D}^n (x) = \mathcal{D}_{u_n} (x) \hbox{ and } u_{n+1} = \Theta_{x,u_n} (1)u_n .
\end{equation}

Thus we compute $u_2$ and $U_2 = \phi^{-1} (\{ u_2 \})$ and proceeding as previously we can algorithmically define $\sigma_{U_2}$, $\delta_{U_2}$, $\Theta_{u_2}$ and $\psi_{u_2}$ exiting the algorithm as it is described above. 
Continuing in this way we compute $\{ u_n | 1\leq n\leq 1+ K_1^{K_1K_2 +2} \}$ and the set

$$
S = \{ (\sigma_{U_n} , \psi_{u_n} )  | U_n= \phi^{-1} (\{u_n\}), 1\leq  n\leq  1+ K_1^{K_1K_2+2}\} 
$$

exiting the algorithm when needed.

If all the elements of $S$ are different (which is algorithmically checkable) then, due to Theorem \ref{theo:derbound}, $x$ cannot be uniformly recurrent.

Suppose there exist $n$ and $m$, $n< m\leq  1+ K_1^{K_1K_2+2}$, such that $(\sigma_{U_n} , \psi_{u_n} ) = (\sigma_{U_m} , \psi_{u_m} )$, then $\mathcal{D}^n (x) = \mathcal{D}^m (x)\in R^\mathbb{N}$ where $R = R_{x,u_n}$.
Consequently $\mathcal{D}^n (x) = \mathcal{D}^{n+k(m-n)} (x)$ for all $k$.
It remains to show that $x$ is uniformly recurrent.

We set $z = \mathcal{D}^n (x)$.
>From \eqref{eq:derivedUR} the sequence $(u_{n+k(m-n)})_{k\in \mathbb{N}}$ is a strictly increasing sequence of prefixes of $x$ and, consequently, there exists $k$ such that for each return word $w$ of $\mathcal{R}_{x,u_n}$, the word $wu_n$ has an occurrence in $u_{n+k(m-n)}$. 
As $u_{n+k(m-n)}$ is a prefix of $w'u_{n+k(m-n)}$ for all $w'\in \mathcal{R}_{x, u_{n+k(m-n)}}$, we deduce that, for each $(w,w')\in \mathcal{R}_{x,u_n} \times \mathcal{R}_{x, u_{n+k(m-n)}}$, 
the word $wu_n$ occurs in $w'u_{n+k(m-n)}$.
In other terms, for all letters $i$ and $j$ in $R$, $i$ occurs in $\tau (j)$ where $\tau $ is the unique morphism satisfying 

\begin{align*}
\label{align:tau}
\Theta_{x,u_n} \tau = \Theta_{x,u_{n+k(m-n)}}.
\end{align*}

This means that the incidence matrix of $\tau$ has positive entries.
Moreover it can be easily checked that all images of $\tau$ start with the letter $1$.
Consequently, $\tau$ is a primitive substitution prolongable on the letter $1$ whose unique fixed point is $z$.
Hence $z$ is uniformly recurrent. 
As $x = \Theta_{x,u_n} (z)$, $x$ is also uniformly recurrent and the uniform recurrence is finally decidable.

\section{Remarks and conclusions}
\label{section:remarks}
\subsection{Morphic sequences versus substitutive sequences}
\label{subsec:morphicsub}

In Section \ref{section:decid} we prove more than what we were searching. 
We incidently obtained Theorem \ref{theo:morphic-sub}.

Let us sketch the proof.

Taking the notation of Section \ref{section:decid}, we got $x = \Theta_{x,u_n} (z)$ and $z=\mathcal{D}^{m-n} (z)$.
Proposition \ref{prop:derder} implies the existence of a prefix $w$ of $z$ such that $\Theta_{z,w} (z) = z$. 
There is no difficulty to check that $\Theta_{z,w}$ is a primitive substitution. 
Thus, we conclude with Proposition 9 in \cite{Durand:1998}.

In fact, Theorem \ref{theo:morphic-sub} was already known in \cite{Durand:1998} (or \cite{Durand:1996} where it is illustrated with the Chacon sequence) for non-erasing morphisms and growing substitutions (see also \cite{Damanik&Lenz:2006}).

Observe that if we do not assume uniform recurrence of the morphic sequence, then it is not necessarily a substitutive sequence. 
For example, let $\sigma$ be the endomorphism defined on $\{ 0, 1\}^*$ by $\sigma (0) = 001$ and $\sigma (1) = 1$.
The sequence $x=\sigma^\infty (0)$ is not uniformly recurrent and is not substitutive (it can not be defined by a growing endomorphism).
For more explanations, let us recall the main result in \cite{Pansiot:1984}.

\begin{theo}
Let $\sigma$ be a non-erasing endomorphism prolongable on the letter $a$.
Let $x=\sigma^\infty (a)$. 
Then,
\begin{itemize}
\item
If $\sigma $ is growing then $(p_x (n))_n$ increases like $1$, $n$, $n\log\log n$ or $n \log n$;
\item
If $\sigma $ is non-growing such that the words consisting of non-growing letters have bounded length, then $x$ is a morphic sequence w.r.t. a substitution and its word complexity grows like in the previous case;
\item
If $\sigma $ is non-growing such that the words consisting of non-growing letters have unbounded length, then $(p_x (n))_n$ increases like $n^2$.
\end{itemize}
\end{theo}

Therefore the word complexity $p_x$ is quadratic.
If $x$ was substitutive, then the underlying endomorphism would be growing and thus its word complexity would increase at most like $n\log n$. This is not possible. 

In \cite{Nicolas&Pritykin:2009} it is observed that uniformly recurrent morphic sequences $x$ have subaffine word complexity. 
This result is a straightforward corollary of Theorem \ref{theo:morphic-sub} but it has a direct and short proof.
Suppose that $x$ is not periodic, then from a result of Pansiot (Lemma \ref{lemma:pansiot}) we can suppose $x$ is morphic w.r.t. a growing substitution $\sigma$.  
As $\sigma $ has a primitive sub-morphism $\tau$, by uniform recurrence we have that $\mathcal{L} (x)$ is the morphic image of the language of a primitive substitution. As primitive substitutions have sub-affine complexity, so has $x$.

\subsection{Derived sequences with respect to words that are not prefixes}

In \cite{Holton&Zamboni:1999} the authors studied the notion of derived sequences not only on prefixes but on all words of the uniformly recurrent sequence $x$ considered. 
Let us call these sequences {\em descendant} of $x$ as they did.
They are defined as follows.
Let $u\in \mathcal{L} (x)$ and $w$ be such that $x=wuy$ where $wu$ has a unique occurrence of $u$. 
The sequence $\mathcal{D}_u (uy)$ is called {\em descendant} of $x$ with respect to $u$. 
We denote it by $\mathcal{D}_u (x)$.

They improved the main result in \cite{Durand:1998} showing the following.

\begin{theo}
\label{theo:HZ}
Let $x$ be a uniformly recurrent sequence.
Then, $x$ is a primitive substitutive sequence if and only if the set of its descendants is finite. 
\end{theo}

A similar result was recently stated in \cite{Kanel-Belov&Mitrofanov:preprint2011} by means of Rauzy graphs.

We think it is important and interesting to recall that C. Holton and L. Zamboni presented this result as 
"a symbolic counterpart to a theorem
of M. Boshernitzan and C. Carroll \cite{Boshernitzan&Carroll:1997} which states that an interval exchange transformation
with lengths in a quadratic field has only finitely many descendants".
Moreover, this last result is presented as an extension of the classical Lagrange's theorem asserting that the continued fraction expansion of  a quadratic irrational is eventually periodic.

Due to Theorem \ref{theo:morphic-sub}, Theorem \ref{theo:HZ} also holds for uniformly recurrent morphic sequences (instead of primitive substitutive sequences).

\begin{theo}
\label{theo:morphicdesc}
Let $x$ be a uniformly recurrent sequence.
Then, $x$ is a morphic sequence if and only if the set of its descendants is finite. 
\end{theo}

\begin{proof}
If $x$ is periodic then it is a morphic sequence and it has finitely many descendants. 
Thus we suppose $x$ is non-periodic.

Let us sketch the proof.
>From Theorem \ref{theo:HZ} (or the main result in \cite{Durand:1998}) we only have to prove that morphic sequences have finitely many descendants. 

Let $x=\phi (y)$ where $y=\sigma^\infty (a)$ and $\sigma $ is an endomorphism on $A^*$ prolongable on the letter $a$. 
Let us first make some remarks.
  
Let $u\in \mathcal{L} (x)$ and $w$ be such that $x=wuz$ where $wu$ has a unique occurrence of $u$.
Then there exists a morphism $\chi_u : R_{x,wu}^* \to R_{uz,u}^*$ such that 

$$
\chi_u (\mathcal{D}_{wu} (x)) =   \mathcal{D}_u (uz).
$$

Moreover, it can be shown that the set $\{ \chi_u | u\in \mathcal{L} (x) \}$ is finite using Theorem \ref{theo:morphic-sub} and Theorem \ref{theo:encad}.
Thus, applying the main result in \cite{Durand:1998} (or Theorem \ref{theo:HZ}) achieves the proof.
\end{proof}

\subsection{Dynamical interpretation}
Let us reformulate Theorem \ref{theo:morphicdesc} and Proposition \ref{prop:boundedpreimages}.
Below $\mathbb{K}$ will stand for either $\mathbb{N}$ or $\mathbb{Z}$.

We refer to \cite{Queffelec:2010} for more details on subshifts.

A {\em (topological) dynamical system} is a couple $(X,S)$ where $S$ is a continuous map from $X$ to $X$ and $X$ is a $S$-invariant compact metric space.
It is a {\em subshift} when $X$ is a closed $S$-invariant subset of $A^\mathbb{K}$ and $S$ is the shift map ($S:A^{\mathbb{K}}\to A^{\mathbb{K}},\ (x_n)_{n\in \mathbb{K}}\mapsto (x_{n+1})_{n\in \mathbb{K}})$.
In the sequel even if the alphabet changes, $S$ we will always denote the shift map.
We say that a dynamical system $(X,S)$ is {\em minimal} whenever the only $S$-invariant subsets of $X$ are $X$ and $\emptyset$.

Let $x$ be a sequence on the alphabet $A$.
Consider the set $\Omega (x) = \{ y\in A^\mathbb{K} | \mathcal{L} (y) \subset \mathcal{L} (x) \}$.
Then, $(\Omega (x) ,S)$ is a subshift.

Let $(X,T)$ and $(Y,R)$ be two dynamical systems.
We say that $(Y,R)$ is a {\em factor} of $(X,T)$ if there is a continuous 
and onto map $\phi: X \rightarrow Y$ such that $\phi \circ T = R \circ \phi$.
If, moreover, $\phi$ is one-to-one we say $\phi $ is a {\em conjugacy}.
Then, Proposition \ref{prop:boundedpreimages} can be interpreted by means of subshifts and factors.

\begin{prop}
Let $x$ be a morphic sequence. 
There exists a constant $K$ such that for any minimal factor $(Y,S)$, with $\# Y=+\infty$, given by $\phi : (\Omega (x) , S) \to (Y,S)$,
for all $y\in Y$, $\# \phi^{-1} (\{ y \}) \leq K$.
\end{prop}

Let $(X,S)$ be a minimal subshift and $U\subset X$ be a clopen set (closed and open set).
We recall that $X$ being a Cantor space, a basis of its topology consists of clopen sets.
For example the family of cylinder sets is such a family. 
A {\em cylinder} is a set of the form $[u.v]_X = \{ x\in X \mid x_{-|u|} \cdots x_{-1} x_0\cdots x_{|v|-1} = uv \}$.
Let $S_U : U\rightarrow U$ be the {\em induced map}, i.e. if $y\in U$ then

\begin{center}
$S_U (y) = S^{r_U(y)} (y) $, where $r_U (y) = \inf \{ n > 0 \ | \  S^n (y) \in U \}$.
\end{center}

As $(X,S)$ is minimal, the map $r_U$ is well-defined and continuous. 
It is called the {\em return map} to $U$. 
The pair $(U , S_U )$ is a minimal dynamical system.
We say $(U , S_U)$ is the {\em induced dynamical system of $(X,S)$ with respect to $U$}.

>From Theorem \ref{theo:morphicdesc} we obtain the following result.

\begin{theo}
Let $x$ be a uniformly recurrent morphic sequence. 
Then, $(\Omega (x) , S)$ has finitely many induced dynamical systems on cylinders up to conjugacy. 
\end{theo}

This could also be directly deduced from \cite{Holton&Zamboni:1999} using the arguments at the end of Section \ref{subsec:morphicsub}.

\subsection{Algorithmical complexity}

In \cite{Nicolas&Pritykin:2009} is shown the decidability of the uniform recurrence for purely morphic sequences and automatic sequences by presenting a polynomial time algorithm.
>From the present work, and for purely morphic sequence, we can deduce a polynomial algorithm as a direct consequence of Lemma \ref{lemme:horn} and Proposition \ref{prop:decomprim}.
But for the general case the algorithm having a high complexity we did not find interesting to compute it precisely.

\subsection{Other decidability problems for morphisms}

In the conclusion of the work \cite{Nicolas&Pritykin:2009} is mentioned some open problems of decidability concerning morphic sequences. 
Let us recall them and add some.
The inputs will always be the morphism $\phi$ and the endomorphism $\sigma$ defining  the morphic sequence or pairs of such couples $(\phi , \sigma )$.

\subsubsection*{Periodicity}

The decidability of the ultimate periodicity has been recently solved in both \cite{Mitrofanov:preprint2011a} and \cite{Durand:preprint2011}. It was previously shown in the primitive case in \cite{Durand:2012} and for purely morphic sequences in \cite{Pansiot:1986} and \cite{Harju&Linna:1986}.

\subsubsection*{Recurrence}

The decidability problem concerning different notions of recurrence such as recurrence or almost periodicity remains open (see \cite{Maes:1998,Maes:2000}, \cite{Muchnik&Semenov&Ushakov:2003} and \cite{Pritykin:2010}). 

\subsubsection*{Language}

Given two morphic sequences, is it decidable whether that they have the same language?

In \cite{Fagnot:1997b} is given some answers to this question for purely morphic sequences generated by primitive substitutions under some mild assumptions.
In \cite{Fagnot:1997a} I. Fagnot solves this decidability problem for automatic sequences. 
She considers two cases : the multiplicative dependence and independence of the dominant eigenvalues $p$ and $q$ (which are integers in these settings) of the underlying substitutions, that is whether $\log p / \log q $ belongs or not to $\mathbb{Q}$.

She managed to treat the multiplicatively independent case using the extension of Cobham's theorem for automatic languages she obtained in the same paper, then testing the ultimate periodicity of the languages and to see if they are equal. 

For the dependent case she used the logical framework of Presburger arithmetic associated to automatic sequences.

Thus, this problem remains open.

\subsubsection*{Equality}

Given two morphic sequences $x$ and $y$, is it decidable whether $x=y$?
This was shown to be true in \cite{Durand:2012} in the primitive case.

For purely morphic sequences, this problem (called the D$0$L $\omega$-equivalence problem) was solved in 1984 by K. Culik II and T. Harju \cite{Culik&Harju:1984}.
More is shown for this problem in \cite{Durand:2012} in the primitive case: the two primitive substitutions have powers that coincide on some cylinders. 

\subsubsection*{Factor and conjugacy}

Let $(X,S)$ and $(Y,S)$ be two subshifts generated by two morphic sequences. 
Could we decide whether one is a factor of the other, or, whether they are conjugate. 
In \cite{Coven&Keane&LeMasurier:2008} some elements of answers can be found for the Morse subshift.

\subsubsection*{Orbit equivalence}

With the same settings as before we can ask for the decidability of the {\em orbit equivalence} of the subshifts. 
That is, does there exist a continuous bijective map $\phi$ sending orbits to orbits: $\phi (\{ S^n (x) | n\in \mathbb{Z} \}) = \{ S^n (\phi (x)) | n\in \mathbb{Z} \})$ for all $x\in X$.
Thanks to the characterization of orbit equivalence by means of dimension groups, a stronger notion of orbit equivalence (called the {\em strong orbit equivalence}) is decidable. 
This comes from \cite{Bratteli&Jorgensen&Kim&Roush:2001} and the fact that the dimension groups of minimal substitution system are stationary (see \cite{Forrest:1997} or \cite{Durand&Host&Skau:1999}).

We refer the interested reader to \cite{Giordano&Putnam&Skau:1995} for more details on (strong) orbit equivalence. 

The decidability of orbit equivalence remains open for substitution subshifts.

\subsection{For tilings}

All these questions could also be asked for self-similar tilings or multidimensional substitutions.
The more tractable problems would certainly concern "square multidimensional substitutions" (see for example the references \cite{Cerny&Gruska:1986,Salon:1986,Salon:1987,Salon:1989} or, in an interesting logical context, the nice survey \cite{Bruyere&Hansel&Michaux&Villemaire:1994}). 
Moreover, the striking papers \cite{Mozes:1989} and \cite{GoodmanStrauss:1998} should also be mentioned as they show that in higher dimension the situation is not as "simple" as it is in dimension 1.

\bibliographystyle{new2}
\bibliography{unifrec}

\end{document}